\documentclass[11pt,centertags]{amsart}
\usepackage{amssymb}            
\usepackage{mathrsfs}             
\usepackage{hyperref}           
\usepackage{ifsym}              
\usepackage{calc}               
\usepackage{enumerate}          
\usepackage[all]{xy}            
\usepackage{psfrag}             
\usepackage[usenames]{color}    
\usepackage{verbatim}           
\usepackage{fancyhdr}           
\usepackage{url}                
\usepackage{gloss}              
\usepackage{yhmath}             

\newif\ifPDF
\ifx\pdfoutput\undefined
    \PDFfalse
\else
    \ifnum\pdfoutput > 0
        \PDFtrue
    \else
        \PDFfalse
    \fi
\fi \ifPDF
    \usepackage[pdftex]{graphicx}
    \DeclareGraphicsExtensions{.pdf,.png,.jpg}
    \graphicspath{{}}
\else
    \usepackage{graphicx}
    \DeclareGraphicsExtensions{.eps}
    \graphicspath{{}}
\fi

\newtheorem*{main*}{Main Theorem}
\newtheorem{main}{Theorem}

\newtheorem{theorem}{Theorem}[section]
\newtheorem*{theorem*}{Theorem}
\newtheorem{proposition}[theorem]{Proposition}
\newtheorem{lemma}[theorem]{Lemma}
\newtheorem{corollary}[theorem]{Corollary}

\newtheorem*{question*}{Question}

\newtheorem*{conjecture*}{Conjecture}

\theoremstyle{definition}

\newtheorem{definition}[theorem]{Definition}
\newtheorem*{definition*}{Definition}

\theoremstyle{remark}

\numberwithin{equation}{section}

\newcommand{\R}{\mathbb{R}}

 \DeclareMathOperator{\Sp}{\Sp}

\newcommand{\union}{\cup}
\newcommand{\intersect}{\cap}

\newcommand{\pa}{\partial }

\providecommand{\to}{\longrightarrow }

\newcommand{\inner}[1]{\left\langle #1 \right\rangle }

\newcommand{\sn}{\operatorname{sn}}

\def\[#1\]{\begin{align*}\begin{split} #1 \end{split}\end{align*} }
\def\[[#1\]]{\begin{align}\begin{split} #1 \end{split}\end{align} }

\renewcommand{\bar}{\overline}

\newcommand{\Fix}{\mathrm{Fix}}

\begin{document}
\author[Khek Lun Harold Chao]{Khek Lun Harold Chao}
\title[The rigidity of some finite group actions  on {CAT}($\kappa$) spaces]{The rigidity of some finite group actions on {CAT}($\kappa$) spaces}
\email{khchao@indiana.edu}
\address{Department of Mathematics, Indiana University, Bloomington, IN 47405, USA}
\begin{abstract}
In this paper, we first prove the optimal lower bound for  Alexandrov angle rigidity of torsion elliptic isometries
on any complete CAT($\kappa$) space, which, when attained, leads to an embedded 2-flat in the tangent cone invariant under the induced action of the isometry.  Next, we will prove similar result for action of symmetry groups  of either a regular orthoplex, a regular hypercube, a regular dodecahedron or a regular icosahedron  on a set of points in any complete CAT($\kappa$) space in a way corresponds to the set of vertices of the polytope, the angle made  at the circumcenter by any pair of points corresponding to an edge is bounded below by that of the edge in the polytope. As a result, we give a condition for the convex hull of the set of points to be isometric to a corresponding regular polytope in a model space of constant curvature $\kappa$.
\end{abstract}
\maketitle

\section{Introduction}
 In CAT($\kappa$) space elliptic isometries one usually deals with are those of finite order. For such an isometry, we would like to know what angle it turns a point that is not fixed. The simplest example is a rotation of order $n$ on a plane. In this case every point except the fixed point is turned by an angle $2\pi/n$, measured at the fixed point. For an orthogonal transformation with finite order in $\R^k$, one can show that  any non-fixed point is turned by an angle at least $2\pi/n$. For $n\geq 3$, this lower bound is achieved iff there is a 2-dimensional invariant subspace such that the action restricted on it is a rotation of angle $2\pi/n$. It is natural to expect the same for a general complete CAT(0) space, i.e. $2\pi/n$ is a lower bound for the angle.

It seems that the above question have received little attention so far. The only result which the author can find is the following estimate by Caprace and Monod using a very short and elementary argument: Let $X$ be a complete CAT(0) space, and $g$ be an elliptic isometry of finite order $n$ on $X$. For any
point $x$ not fixed by $g$, denote by $c$ the closest point to
$x$ in $\Fix(g)$. (Then $c$ has to be the circumcenter of the orbit of $x$.) In the middle of their proof of Alexandrov angle rigidity (\cite{CM1}, Proof of Proposition 6.8), they showed that $\angle_{c}(x, g\cdot x) \geq  1/n$. This lower bound is much smaller than the one we want. (However, it must be noted that the focus of that paper is not on metric geometric side but on group theoretic side, thus an optimal bound was not needed therein.)

In this paper, we will  show that $2\pi/n$ is a lower bound for a finite order elliptic isometry in a complete CAT(0) space and, more generally, in a complete CAT($\kappa$) space. We suppose $n \geq 3$, since for $n=2$ the inequality is trivial and is actually an equality.
\begin{main}\label{inequality}
Let $g$ be any elliptic isometry with finite order $n \geq 3$ on a complete CAT($\kappa$) space $X$, and $x$ be a point not fixed by $g$. If $\kappa >0$ assume that the orbit of $x$ by $g$ has radius less than $\pi/(2\sqrt \kappa)$. Then
$$\angle_{c}(x, g\cdot x) \geq \frac{2\pi}{n}$$
where $c$ is the circumcenter of the orbit of $x$.
\end{main}
Moreover, analogous to the situation in $\R^k$, we have an invariant 2-flat in the tangent cone when this bound is achieved.
\begin{main}\label{equality}
If $n \geq 3$ and  $\angle_c (x,g\cdot x) = 2\pi /n$, then the concatenation of the geodesic segments $[g^i \cdot \chi, g^{i+1} \cdot \chi]$ in $\overline{S_p(X)}$ forms an isometrically embedded circle of length $2\pi$, where $\chi$ is the direction of geodesic from $p$ to $x$. Also the tangent cone at $c$ contains a 2-flat on which $g$ induces a rotation of angle $2\pi/n$.
\end{main}

Using these results and the technique in \cite{LangSchroeder}, we will prove inequalities of Alexandrov angles similar to the one above for the symmetry groups of  Platonic solids, regular orthoplex and regular hypercube acting on a set of vertices in any complete CAT($\kappa$) space. The case of the regular  $n$-simplex is actually a special case of Theorem A in \cite{LangSchroeder}.

\begin{main}\label{polyhedron}
Let $W$ be a regular Platonic solid or a regular hypercube or a regular orthoplex with vertices $\bar{x_i}$, and $G$ be the (orientation preserving) symmetric group of $W$. Suppose that $G$ acts on a CAT($\kappa$) space $X$ by isometries.  Let $x_i$ be points in $X$ on which the induced action of $G$ is equivariant with that on $\bar {x_i}$. Suppose that the set of points $\{x_i\}$ has radius less than $\pi/(2\kappa)$ if $\kappa>0$. For any edge of $W$ with endpoints $\bar {x_i}$ and $\bar {x_j}$, the Alexandrov angle $\angle_c(x_i, x_j)$ is no less than the corresponding (Euclidean) angle $\angle_{\bar c}(\bar {x_i}, \bar {x_j})$, where $c$ and $\bar c$ are the circumcenters of the set $\{x_i\}$ and $W$ respectively. As a consequence, we have $\sn_\kappa a/2 \ge \sn_\kappa r \sin(\alpha/2)$, and if equality holds, then the convex hull of $\bar{x_i}$ is isometric to a corresponding regular polytope in a model space of constant curvature $\kappa$.
\end{main}

 The main tool of \cite{LangSchroeder} is a generalized ``scalar product'' defined on a  complete CAT(0) 0-cone , of which the tangent cone at a point in a  CAT($\kappa$) space is an example. The CAT(0) inequality makes this scalar product concave, which is a key property that enables us to derive the lower bound inequality. We will also need an inequality that results from the lower bound for  elliptic isometries  in the first part, which gives relationships between distances of different pairs of points in the set $\{x_i\}$, to reduce the inequality obtained from the scalar product calculation to only one variable.

\section{Proof of the lower bound}

We recall the definition of
space of directions (\cite{BH} Def. II.3.18).
\begin{definition}
For any $p \in X$, consider all the non-trivial geodesics issuing
from $p$. Define an equivalent relation on the set by $\gamma \sim \gamma'$
iff the Alexandrov angle between them $\angle_p(\gamma,\gamma')=0$. The
equivalent classes under this relation form a metric space with
Alexandrov angle $\angle(\cdot, \cdot)$ as the metric. This metric
space is the \textit{space of directions} at $p$ and is  denoted as $S_p(X)$.
\end{definition}

By a theorem of Nikolaev (\cite{BH} Theorem II.3.19), as long as $X$ is a metric space with curvature bounded above by some $\kappa$, then the completion of $S_p(X)$ is a CAT(1) space.

The following is a simple lemma. We include a proof which will be referred to later.
\begin{lemma}\label{short closed curve}
Let $\alpha$ be a closed curve on the unit 2-sphere. If $\alpha$ is
shorter than 2$\pi$, then there is a point $m$ on the sphere such
that $\alpha$ is contained in the hemisphere $B(m,\pi/2)$, where the
metric is the angle metric on the sphere.
\end{lemma}
\begin{proof}
Take two points $x, x'$ on the curve $\alpha$ that divide it into two curves of equal length. As $d(x,x') < \pi$,  there is a unique geodesic segment joining $x$ and $x'$, and we let $m$ be the midpoint of this segment. Then $m$ satisfies the desired condition. Otherwise, there would be a point $z$ on curve $\alpha$ with $d(z,m)= \pi/2$. Call the subcurve of $\alpha$ on which $z$ lies as  $\alpha_1$. Rotate the sphere for an angle $\pi$ about the axis through $m$, and let $z'$, $\alpha'_1$ be the rotated images of $z$ and $\alpha_1$. Now $d(z,z')=2d(z,m) = \pi$, but there is a path on $\alpha_1 \union \alpha'_1$ joining $z$ to $z'$ with length equal that of $\alpha_1$, which is shorter than $\pi$, a contradiction.
\end{proof}

The main ingredient of our proof is the majorization theorem of
Reshetnyak.

\begin{theorem}[Reshetnyak majorization theorem \cite{AlexGeom}]
Any closed curve $\alpha$ in a CAT($\kappa$) space $U$ with length less than $2\pi / \sqrt \kappa$ is majorized by a convex region $D$ in $M^2(\kappa)$, i.e. there is a distance non-increasing map from $D$ to $U$ such that its restriction to the boundary $\pa D$ is mapped to $\alpha$ preserving the length.
\end{theorem}

For $\kappa\leq 0$, define a map $\exp^{-1}_p : X\setminus\{p\} \to \overline{S_p(X)}$ which maps every point $x$ different from $p$ to the direction represented by the geodesic segment $[p, x]$; for $\kappa > 0$, define  $\exp^{-1}_p$ similarly but with domain $B(p, \pi/\sqrt \kappa)\setminus\{p\}$.

 Let $K$ be a compact set in $X$. If $\kappa > 0$ assume that $K$ has radius less than $\pi/(2\sqrt \kappa)$. By Theorem B of \cite{LangSchroeder} any bounded set of radius  less than $\pi/(2\sqrt \kappa)$ in a complete CAT($\kappa$) space has a unique circumcenter ; while if  $\kappa \leq 0$ the uniqueness of circumcenter of any bounded set is automatic. Let $p$ be a point in $X$ such that $K \neq \{p\}$, and let $r= \sup_{x\in K} d(p,x)$.
\begin{proposition}\label{circumcenter}
If the image of $\pa B(p,r) \intersect K$ under $\exp^{-1}_p$ is contained
in $B(m,\pi/2)$ for some point $m\in \overline{S_p(X)}$, then $p$ is
not the circumcenter of $K$.
\end{proposition}

\begin{proof}
As $m$ is in the completion of the space of directions $S_p(X)$, with slight modification if necessary, $m$ is represented by a geodesic $\gamma([0,s_1])$, where $\gamma(0)=p$, $s_1 > 0$, and $s_1 < \pi/(2\sqrt \kappa) - r$ if $\kappa >0$, so that $\gamma \subset B(x,\pi/(2\sqrt \kappa))$ for any point $x \in K$. Then $s \mapsto d(x,\gamma(s))$ is convex for any $x\in K$.

We claim that there exists some point $\gamma(s_2)$ such that $K \subset B(\gamma(s_2), r)$, implying that $p$ is not the circumcenter of $K$. Suppose the claim is false, then for every $\gamma(s_1/n)$ there is a point $x_n \in K$ with $d(x_n,\gamma(s_1/n)) \ge r$. Then by convexity of $s \mapsto d(x_n,\gamma(s))$ the function is strictly increasing on $[s_1/n, s_1]$. Since $K$ is compact, a subsequence of $x_n$,  denoted again by $s_n$, converges to a point $x_0$, then the functions $s \mapsto d(x_n,\gamma(s))$ converges uniformly to $s \mapsto d(x_0,\gamma(s))$. So $r \ge d(x_0, p) \ge \liminf_n d(x_n,\gamma(s_1/n) \ge r$, i.e. $d(x_0, p) =r$, and $s \mapsto d(x_0,\gamma(s))$ is strictly increasing on $[0,s_1]$. But from the first variation
formula for CAT($\kappa$) space (\cite{BH}, Corollary II.3.6),
$$\cos \angle_p(\gamma(\cdot),x_0)=\lim_{s\to 0}\frac {d(p,x_0)-d(\gamma(s),x_0)}{s} > 0,$$
a contradiction. Hence the claim.\end{proof}

With the above results we show Theorem \ref{inequality}.
\begin{proof}
Note that Theorem B of \cite{LangSchroeder} asserts that the orbit of $x$ has a unique circumcenter for the case $\kappa >0$, so the point $c$ is well-defined. Suppose on the contrary that $\angle_{c}(x, g\cdot x) <  2\pi /n$. Joining $\exp^{-1}_p (g^i \cdot x)$ successively by geodesic segments in $\overline{S_c(X)}$, we obtain a closed $n$-gon in $\overline{S_{c}(X)}$ with length less than $2\pi$. Apply the Reshetnyak majorization theorem, we get a distance non-increasing map from a convex region on the unit 2-sphere to $\overline{S_c(X)}$ such that the boundary of the convex region is mapped to the $n$-gon both of which have the same length. By Lemma \ref{short closed curve}, there exists a point $\bar m$ on the 2-sphere such that $B(\bar m,\pi/2)$ covers the convex region, and from the construction of $\bar m$ in the proof it follows that $\bar m$ is in the convex region. Let $m$ be the image of $\bar m$ in $\overline{S_c(X)}$. As the map does not increase distance, the $n$-gon is in $B(m,\pi/2)$.
Proposition \ref{circumcenter} implies that $c$ is not the circumcenter, a contradiction.
\end{proof}

Next we show Theorem \ref{equality} by using Theorem \ref{inequality} to derive a contradiction.
\begin{proof}
We claim that joining any two consecutive segments gives a local geodesic. For this it suffices to show that for any $m \in [\chi, g\cdot \chi]$, $d(m,g\cdot m)=2\pi/n$. Suppose not, then take $m \in [\chi, g\cdot \chi]$ such that $d(m,g\cdot m)< 2\pi/n$. The segments $[g^i \cdot m, g^{i+1}\cdot m]$ form a curve shorter than $2\pi$, so using Reshetnyak majorization theorem as before this curve is contained in an open ball of radius $\pi/2$. Since $\overline{S_c(X)}$ is CAT(1), the orbit of $m$ has a unique circumcenter $q$ in $\overline{S_c(X)}$, which is fixed by $g$. The distance between $q$ and the orbit of $\chi$ equals $d(q, \chi)$, and by Proposition \ref{circumcenter} this distance is at least $\pi/2$, otherwise $c$ could not be the circumcenter of the orbit of $x$, so there exists $\chi'\in [\chi,m]$ with $d(\chi',q)=\pi/2$.

Applying Theorem \ref{inequality} to $\overline{S_c(X)}$ and the orbit of $m$, we see that $\angle_q(m,g\cdot m)\geq 2\pi/n$. Then
\begin{align*}
\angle_q(\chi',m) + \angle_q(m,g\cdot \chi') & = \angle_q(g\cdot \chi', g\cdot m)+\angle_q(m, g\cdot \chi') \\
& \geq  \angle_q(m, g\cdot m)\geq \frac {2\pi} n,
\end{align*}
while
\begin{align*}
d(\chi',m) + d(m,g\cdot \chi') & \leq d(\chi',m) + d(m,g\cdot \chi) + d(g\cdot \chi,g \cdot \chi') \\
&= d(\chi',m) + d(m,g\cdot \chi) + d (\chi, \chi')\\
&=d(\chi, g\cdot \chi) = \frac {2\pi} n.
\end{align*}
Construct the comparison triangles $\triangle (\bar q,\bar{\chi'},\bar
m)$ of $\triangle (q,\chi',m)$ and $\triangle (\bar q, \bar m,\bar {g\cdot \chi'})$ of $\triangle (q,m,g\cdot \chi')$ on the unit 2-sphere on
the opposite sides of $[\bar q, \bar m]$. Since $d(\bar q,\bar m) < \pi
/2$, the segments $[\bar{\chi'},\bar m]$ and $[\bar m, \bar{g\cdot
\chi'}]$ do not form a geodesic, so
$$d(\bar{\chi'}, \bar{g\cdot \chi'}) < d(\bar{\chi'},\bar m) + d(\bar m,\bar{g\cdot \chi'}) =d(\chi',m) + d(m,g\cdot \chi') \leq \frac {2\pi} n,$$
 As $d(\bar q,\bar{\chi'}) = d(\bar q,\bar {g\cdot \chi'}) =\pi/2$, in the triangle $\triangle(\bar q,\bar{\chi'},\bar{g\cdot \chi'})$ we have $\angle_{\bar q}(\bar{\chi'},\bar{g\cdot \chi'}) = d(\bar{\chi'},\bar{g\cdot \chi'}) < 2\pi/n$.
But we have
\begin{align*}
\angle_{\bar q}(\bar{\chi'},\bar{g\cdot \chi'}) &=\angle_{\bar q}(\bar{\chi'},\bar m) +\angle_{\bar q}(\bar m,\bar {g\cdot \chi'}) \\
&\geq \angle_q(\chi',m) + \angle_q(m,g\cdot \chi') \geq \frac {2\pi} n.
\end{align*}
Hence a contradiction. (Note that here if we used $\chi$ instead of $\chi'$, then $d(\bar q,\bar{\chi}) \geq \pi/2$ only implies $\angle_{\bar q}(\bar{\chi},\bar{g\cdot \chi}) \geq d(\bar{\chi},\bar{g\cdot \chi})$, and the argument would not work.)

The claim means that the concatenation of  $[g^i \cdot \chi, g^{i+1} \cdot \chi]$ is a local geodesic in the CAT(1) space $\overline{S_c(X)}$, thus it is an isometric embedding of a circle of length $2\pi$.

The Euclidean cone over this circle is a 2-flat in the tangent cone. Since $g$ acts on the circle by rotation of $2\pi/n$, this gives the same rotation by $g$ on the tangent cone.
\end{proof}

Before closing this section, we give an inequality on distances between 3 consecutive orbit points in CAT(0) space using the above inequality. We suppose the order of $g$ is at least 4, since if the order is 3, the orbit points must form an equilateral triangle.
\begin{corollary}\label{polyineq}
Let $X$ be a complete CAT(0) space. Suppose $g$ has order $n \geq 4$. For any point $x$ not fixed by $g$,
$$d(g^2 \cdot x, x ) \leq 2 \cos \left(\frac \pi n \right) d(g \cdot x, x),$$
with equality holds iff $g^i \cdot x$ are vertices of an isometrically embedded flat regular $n$-gon.
\end{corollary}
\begin{proof}
To simplify notation we denote $g\cdot x$ and $g^2\cdot x$ by $y$ and $z$. Consider the two comparison triangles $\triangle(\bar c, \bar x, \bar y)$ and $\triangle(\bar c, \bar y, \bar z)$ on the flat plane with $\bar {c}$ and $\bar{y}$ as their common vertices, and place $\bar x$ and $\bar{z}$ on opposite sides of edge $[\bar c, \bar y]$. These two triangles are congruent isosceles triangles.

In the case $\angle_{\bar c}(\bar x,\bar y) > \pi/2$, we have strict inequality
\begin{align*}
d(g^2\cdot x, x) &\leq d(g^2\cdot x, c) + d(c,x) = d(\bar z,\bar c)+d(\bar c,\bar x) \\
&< \sqrt 2  d(\bar x, \bar y) =  \sqrt 2 d(g \cdot x, x)\leq 2 \cos\left( \frac \pi n \right) d(g \cdot x, x).
\end{align*}

In the case $\angle_{\bar c}(\bar x,\bar y) \leq \pi/2$, draw segment $[\bar x, \bar z]$ intersecting segment $[\bar c, \bar y]$ at the point $\bar p$. By Theorem \ref{inequality} we get
\begin{align}\label{eq:1a}
\angle_{\bar c}(\bar x,\bar y) \geq \angle_c(x,y) \geq \frac {2\pi}{n},
\end{align}
hence
\begin{align}\label{eq:1b}
\angle_{\bar x}(\bar y, \bar p) = \frac 1 2\angle_{\bar c}(\bar x,\bar y) \geq \frac {\pi}{n},
\end{align}
 and
\begin{align}\label{eq:2}
d(\bar p, \bar x) =d(\bar z,\bar p)= d(\bar x, \bar y) \cos
\angle_{\bar x}(\bar y, \bar p) \leq \cos \left(\frac{\pi}{n}\right)
d(\bar x, \bar y)
\end{align}
Let $p$ be on geodesic segment $[c,y]$ such that $d(c,p) = d(\bar c,\bar p)$. It follows that
\begin{align}\label{eq:3}
  \begin{aligned}
d(g^2\cdot x, x) &\leq d(g^2\cdot x,p) + d(p,x) \leq d(\bar z,\bar p) + d(\bar p,\bar x) \\
&\leq  2 \cos \left(\frac{\pi}{n}\right) d(\bar x, \bar y) =  2 \cos
\left(\frac{\pi}{n}\right)d(g \cdot x, x)
  \end{aligned}
\end{align}
Hence the inequality.

If the equality holds, by (\ref{eq:2}) and (\ref{eq:3}) $\angle_{\bar x}(\bar y, \bar p) =\pi/n$, so from (\ref{eq:1a}) and (\ref{eq:1b}) we have $\angle_{\bar c}(\bar x,\bar y) =\angle_c(x,y) =2\pi/n$. Then the Flat Triangle Theorem implies that $\triangle(c,x,g\cdot x)$ is flat, so is $\triangle(c,g\cdot x,g^2 \cdot x)$. Moreover, from  (\ref{eq:3}) $d(g^2\cdot x, x) =d(g^2\cdot x,p) + d(p,x)$, so $[g^2\cdot x,p]\union [p,x]$ is a geodesic, which implies that $\triangle(x,g\cdot x, g^2\cdot x)$ is flat. Hence $\angle_{g\cdot x}(x,g^2\cdot x) = (n-2)\pi/n$, thus the sum of the angles of the $n$-gon equals $(n-2)\pi$, so by a corollary of the Flat Quadrilateral Theorem (\cite{BH} Exercise II.2.12(1)), the convex hull of the $n$-gon is isometric to a flat regular $n$-gon. The converse is clear.
\end{proof}

\section{Regular polytopes}
We will now derive results for a finite set of points with symmetry of any regular polyhedron similar to Theorem \ref{inequality}. We will employ the tools for tangent cones developed by Lang and Schroeder \cite{LangSchroeder} and we will also need the inequality in Theorem \ref{equality}. The main theorem we will prove is the following:

\begin{theorem}\label{anglepolyhedra}
Let $W$ be a regular Platonic solid or a regular hypercube or a regular orthoplex with vertices $\bar{x_i}$, and $G$ be the (orientation preserving) symmetric group of $W$. Suppose that $G$ also acts on a CAT($\kappa$) space $X$ by isometries.  Let $x_i$ be points in $X$ on which the induced action of $G$ is equivariant with that on $\bar {x_i}$. Suppose that the set of points $\{x_i\}$ has radius less than $\pi/(2\kappa)$ if $\kappa>0$. For any edge of $W$ with endpoints $\bar {x_i}$ and $\bar {x_j}$, the Alexandrov angle $\angle_c(x_i, x_j)$ is no less than the corresponding (Euclidean) angle $\angle_{\bar c}(\bar {x_i}, \bar {x_j})$, where $c$ and $\bar c$ are the circumcenters of the set $\{x_i\}$ and $W$ respectively. In the equality case the convex hull of $\{x_i\}$ is isometric to the regular polytope $W$ with radius 1.
\end{theorem}

We note that if $W$ is a regular $n$-simplex, then the above theorem is a special case of the main theorem of Lang and Schroeder (\cite{LangSchroeder} Theorem A).

Project the points $x_i$ to the completed space of directions $\overline{S_c(X)}$, then embed $\overline{S_c(X)}$ to its tangent cone, which is the Euclidean cone $C_0(\overline{S_c(X)})$. Let $v_i$ be the images of points $x_i$ in the tangent cone, and let $o$ be the origin of the tangent cone. Then $\overline{S_c(X)}$ is the unit sphere centered at $o$, and $v_i$ are at a distance 1 away from $o$. The action of $G$ on $X$ induces one on $C_0(\overline{S_c(X)})$ by isometries.

The following has been noted in \cite{LangSchroeder}.
\begin{lemma}
The origin $o$ is the circumcenter of the points $v_i$.
\end{lemma}
\begin{proof}
Assume otherwise, then the segment from $o$ to the circumcenter of the points $v_i$ would make an angle less than $\pi/2$ with segments from $o$ to $v_i$. With slight perturbation if necessary, this direction corresponds to a geodesic segment from $o$ along which the distance to $x_i$ decreases, hence $o$ would not be the circumcenter of $x_i$, a contradiction.
\end{proof}

Thus it suffices to prove the theorem for tangent cones. In the following, we let $Y$ be a metric space such that the Euclidean cone $C_0(Y)$ is a complete CAT(0) space.

Recall that the ``scalar product'' on $C_0(Y)$ is defined as
$$ \inner{ v, w} :=\|v\| \|w\| \cos \angle_o(v,w) $$
with the concavity property
$$\langle \gamma(t), w\rangle \ge (1-t)\inner{u,w} + t \inner{v,w}$$
resulting from the CAT(0) inequality, where $\gamma:[0,1] \to  C_0(Y)$ is a geodesic segment with $\gamma(0)=u$ and $\gamma(1)=v$.

\begin{proposition}[\cite{LangSchroeder} Proposition 2.4]\label{LSineq}
Let $v_1, \cdots , v_n\in C_0(Y)$, and $(v,\lambda), (v',\lambda')\in\mathcal{C}$, where $\lambda=(\lambda_1,\cdots, \lambda_n)$ and $\lambda'=(\lambda'_1,\cdots, \lambda'_n)$. Then $\inner{v,v'} \geq \sum_{i,j=1}^n \lambda_i\lambda'_j\inner{v_i,v_j}$.
\end{proposition}

For a convex hull $K$ of $n$ points $v_i$, a correspondance $\mathcal{C}\subset K\times \Delta_{n-1}$ with the set of $n$-tuples $\Delta_{n-1}=\{(\lambda_1,\lambda_2,\cdots,\lambda_n) : 0 \leq \lambda_i \leq 1, \sum_{i=1}^{n} \lambda_i=1\}$ is defined in \cite{LangSchroeder} as follows
\begin{enumerate}
  \item $(v_1,e_i)\in C$, where $e_i$ is the $i$-th unit vector in the standard basis of $\R^{n}$
  \item For any $(v,\lambda), (v',\lambda')\in \mathcal{C}$, let $\gamma:[0,1]\to C_0(Y)$ be a geodesic from $v$ to $v'$, then $(\gamma(t),(1-t)\lambda + t\lambda') \in \mathcal{C}$ for all $t\in[0,1]$
\end{enumerate}
The projections of this correspondence to $K$ and $\Delta_n$ are surjective but may not be injective.

If a finite group $G$ acts on the points $v_i$, then it induces an action on $\mathcal{C}$ by
$g \cdot (v,(\lambda_1,\cdots, \lambda_n)) = (g\cdot v,(\lambda_{\sigma(1)}, \cdots , \lambda_{\sigma(n)}))$
where $\sigma$ is a permutation on the indices induced by $g$ such that $\sigma(j) = i $ when $g\cdot v_i = v_j$. It can be seen that $(g\cdot v,(\lambda_{\sigma(1)}, \cdots , \lambda_{\sigma(n)}) \in \mathcal{C}$ for any $g\in G$ and $(v,(\lambda_1,\cdots, \lambda_n)) \in\mathcal{C}$.

Suppose $o$ is the circumcenter of the points $v_i$, then it is in the closure $\bar K$ of the convex hull, so for any $\epsilon > 0$ there exists a point $p\in K$ such that $\| p\| < \epsilon$, thus $\mathcal{C}$ contains $(p,(\lambda_1,\cdots \lambda_n))$. Since the action of $G$ on $C_0(Y)$ stabilized the set $\{v_i\}$, it must fix the unique circumcenter $o$. Hence $\|g\cdot p\| = \|p\|$ for any $g$. Let $G=\{g_j\}_1^m$. Consider $g_j \cdot (p,(\lambda_1,\cdots ,\lambda_n)) =(g_j \cdot p,(\lambda_{\sigma_j(1)},\cdots ,\lambda_{\sigma_j(n)})) $. In the collection of $k$-th coordinates $\sigma_j(k)$, every $\lambda_i$ appears the same number of times, so the average
$$\frac 1 m \sum_{j=1}^m (\lambda_{\sigma_j(1)},\cdots ,\lambda_{\sigma_j(n)})) =\left (\frac 1 n,\cdots \frac 1 n \right )$$

Now define successively $p_1 = g_1\cdot p$, $p_{k+1} = \gamma_k(\frac 1 {k+1})$ where $ \gamma_k:[0,1]\to C_0(Y)$ is a geodesic from $p_k$ to $g_{k+1}\cdot p$. $p_1$ has as its corresponding $n$-tuple $(\lambda_{\sigma_1(1)},\cdots ,\lambda_{\sigma_1(n)}))$.  Suppose the corresponding $n$-tuple of $p_k$ is $\frac 1 k\sum_{j=1}^k (\lambda_{\sigma_j(1)},\cdots ,\lambda_{\sigma_j(n)})$, i.e. $(p_k,\frac 1 k\sum_{j=1}^k (\lambda_{\sigma_j(1)},\cdots ,\lambda_{\sigma_j(n)})) \in \mathcal{C}$, then the corresponding $n$-tuple of $p_{k+1}$ is
\begin{align*}
& \frac k {k+1} \cdot \frac 1 k\sum_{j=1}^k (\lambda_{\sigma_j(1)},\cdots ,\lambda_{\sigma_j(n)}) + \frac 1 {k+1} (\lambda_{\sigma_{k+1}(1)},\cdots ,\lambda_{\sigma_{k+1}(n)})  \\
=& \frac 1 {k+1} \sum_{j=1}^{k+1} (\lambda_{\sigma_j(1)},\cdots ,\lambda_{\sigma_j(n)})
\end{align*}
Thus $(p_k,\frac 1 k\sum_{j=1}^k (\lambda_{\sigma_j(1)},\cdots ,\lambda_{\sigma_j(n)})) \in \mathcal{C}$ for all $k$, so $\left(p_m, \left(\frac 1 n,\cdots \frac 1 n \right) \right)\in \mathcal{C}$. If $\|p_k\|, \|g_{k+1}\cdot p\| <\epsilon$, then the CAT(0) inequality on $\triangle(0,p_k,g_{k+1}\cdot p)$ implies that any point on the segment $[p_k,g_{k+1}\cdot p]$ has norm less than $\epsilon$. Therefore we have $\|p_k\| <\epsilon$ for all $k$. Applying Proposition \ref{LSineq} to $\left(p_m,\left (\frac 1 n,\cdots \frac 1 n \right)\right)$, we have
$\epsilon^2 > \|p_m\|^2 \ge \sum_{i,j=1}^n (\frac 1 n)^2 \inner{v_i,v_j}$
Thus
\begin{align}\label{ineq}
0 \ge \sum_{i,j=1}^n \inner{v_i,v_j}
\end{align}

Let  $G$ be the symmetric group of $W$ acting on $C_0(Y)$ by isometries.  If $g \cdot [\bar {x_i},\bar {x_j}] =[\bar {x_{i'}},\bar {x_{j'}}]$ for (unoriented) segments $[\bar {x_i},\bar {x_j}]$ and $[\bar {x_{i'}},\bar {x_{j'}}]$ in $W$, then $g\cdot [v_i,v_j] = [v_{i'},v_{j'}]$ for segments $[v_i,v_j]$ and $[v_{i'},v_{j'}]$ in $C_0(Y)$, so $\inner{v_i,v_j}=\inner{v_{i'},v_{j'}}$.

We are going to prove Theorem \ref{anglepolyhedra} for each case of $W$ ,but we will leave the equality case to the next section.

\begin{proof}[Proof for orthoplex]
Suppose that $W$ is a $k$-dimensional orthoplex. Consider the  $k$-dimensional orthoplex in $\R^k$ with vertices $\{\pm e_i\}_{i=1}^k$, where $\{e_i\}$ is the standard basis of $\R^k$. There are two orbits of chords under the symmetry group action of the orthoplex, one consists of pairs of opposite vertices $(e_i, -e_i)$, another consists of the edges. Label the vertices of the orthoplex as$x_i=e_i$ and $x_{k+i} = -e_i$ for $1\le i \le k$, and label the $2k$ points in $C_0(Y)$ corresponding to $x_i$ and $x_{k+i}$ as $v_i$ and $v_{k+i}$ respectively, on which $G$ acts in the same way as on the vertices of the orthoplex.

From the inequality \ref{ineq},
\begin{align*}
0 &\ge \sum_{i,j=1}^{2k} \inner{v_i,v_j} = 2k+ 2k(2k-2) \inner{v_1,v_2} + 2k\inner{v_1,v_{n+1}} \\
&= 2k(1+ (2k-2) \cos\angle_o(v_1,v_2) + \cos\angle_o(v_1,v_{n+1})) \\ &\ge 2k(1+ (2k-2) \cos\angle_o(v_1,v_2) +(-1)) \\
&=2k(2k-2) \cos\angle_o(v_1,v_2)
\end{align*}
Hence $\angle_o(v_1,v_2) \ge \frac \pi 2 = \angle_{\bar c}(x_1,x_2)$.
\end{proof}

\begin{proof}[Proof for hypercube]
Suppose $W$ is a $k$-dimensional hypercube. Consider the $k$-dimensional  hypercube in $\R^k$ with $2^k$ vertices $v_i=(t^{(i)}_1,\cdots ,t^{(i)}_k)$ where each $t^{(i)}_i$ is either $\pm 1$. There are $k$ orbits of chords under the symmetry group action of the hypercube; each class consists of pairs of vertices differing in the same number of coordinates. Edges of the hypercubes are pairs of vertices differing in only one coordinate. These edges make an angle $\arccos (1-2/k)$ at the center. Assume there are $2^k$ points $v_i$ in $C_0(Y)$ acted on in the same way by the group $G$. For any two points corresponding to two vertices of the hypercube differing in $j$ coordinates, where $1 \le j \le k$, let $a_j$ be the distance between them and $\alpha_j$ be the angle they make at the circumcenter $o$.

We will need the following simple inequality.
\begin{lemma}
Let $X$ be a complete CAT(0) space, $\{x_i\}$ be a finite set of points in $X$, $c$ be the circumcenter of $\{x_i\}$, and $\angle_c(x_1,x_2) \ge \alpha$. If all the points $x_i$ are at an equal distance from $c$, then $$d(x_1,x_2)^2 \ge \operatorname{diam}(\{x_i\})^2 (1-\cos \alpha)/2.$$
\end{lemma}
\begin{proof}
Let $d(x_i,c) = r$. Using the comparison triangle, we see that $$d(x_1,x_2)^2 \ge 2 r^2 (1-\cos \alpha).$$ Since $\operatorname{diam}(\{x_i\}) \le 2r$, we have the result.
\end{proof}

We will do an induction on the dimension $k$ to prove that $\alpha_1 \ge \arccos(1-2/k)$. The case $k=2$ has already been proved in Corollary \ref{polyineq}. Assume that the assertion is true for $k-1$. For $2 \le m < k$, the symmetry group of the $m$-hypercube is embedded in that of the $k$-hypercube as a subgroup acting on a $m$-dimensional subspace of $\R^k$. By induction hypothesis, the vertices of a $m$-dimensional face of the hypercube satisfy the angle inequality for the $m$-hypercube. Let $V_m$ be the vertex set of this hypercube, $v_i$ and $v_j$ be an edge of it. The above lemma gives
\begin{align*}
a_1^2 &=d(v_i,v_j)^2\ge \operatorname{diam}(V_m)^2 (1-\cos \alpha_1)/2 \\
&\ge a_m^2 (1-\cos(\arccos(1-2/m))/2 = a_m^2 /m.
\end{align*}

Since triangles $\triangle(o,v_i,v_j)$ are flat in $C_0(Y)$, if $d(v_i,v_j)=a_m$ and $\angle_o(v_i,v_j)=\alpha_m$, then by cosine law $\cos\alpha_m=(1-a_m/2)$. From the inequality \ref{ineq},
\begin{align*}
0 &\ge \sum_{i,j=1}^{2^k} \inner{v_i,v_j} =2^k + \sum_{j=1}^k 2^k \binom k j \cos \alpha_j \\
&=2^k + \sum_{j=1}^k 2^k \binom k j \left (1-\frac {a_j^2} 2\right) \\
&\ge 2^k + \sum_{j=1}^k 2^k \binom k j \left (1-\frac {j a_1^2} 2 \right) \\
&=2^k(2^k-k \cdot 2^{k-2}a_1^2)
\end{align*}
Hence $a_1^2 \ge 4/k$, and so $\alpha_1 \ge \arccos(1-2/k)$, which is the corresponding angle in the $k$-dimensional hypercube.
\end{proof}

\begin{proof}[Proof for icosahedron]
Suppose $W$ is an icosahedron. There are three orbits of chords in the regular icosahedron, the representatives of which are $(\bar{x_0},\bar{x_1})$, $(\bar{x_0},\bar{x_2})$, $(\bar{x_0},\bar{x_3})$, as shown in Figure \ref{ico1}. Denote the lengths of $(v_0,v_1)$, $(v_0,v_2)$, $(v_0,v_3)$ in $C_0(Y)$ as  $a_1, a_2, a_3$ respectively, and the angles they make with the circumcenter as $\alpha_1, \alpha_2, \alpha_3$ respectively. The stabilizer group $C_5$ of $v_{h_0}$ has an orbit  $\{v_{h_1},v_{h_2},v_{h_3},v_{h_4},v_{h_5}\}$ in Figure \ref{ico2}. (We illustrate in this figure only the corresponding vertices $\bar{x_i}$ in $W$ of $v_i$ but omit the vertices $v_i$ in $C_0(Y)$, and we will do the same for figures in the next case.) Then we have $a_2 =d(v_{h_1},v_{h_3}) \le 2 \cos(\pi/5) a_1$ by Corollary \ref{polyineq}.

\begin{figure}[h]
\centering
\def\svgwidth{0.3\columnwidth}
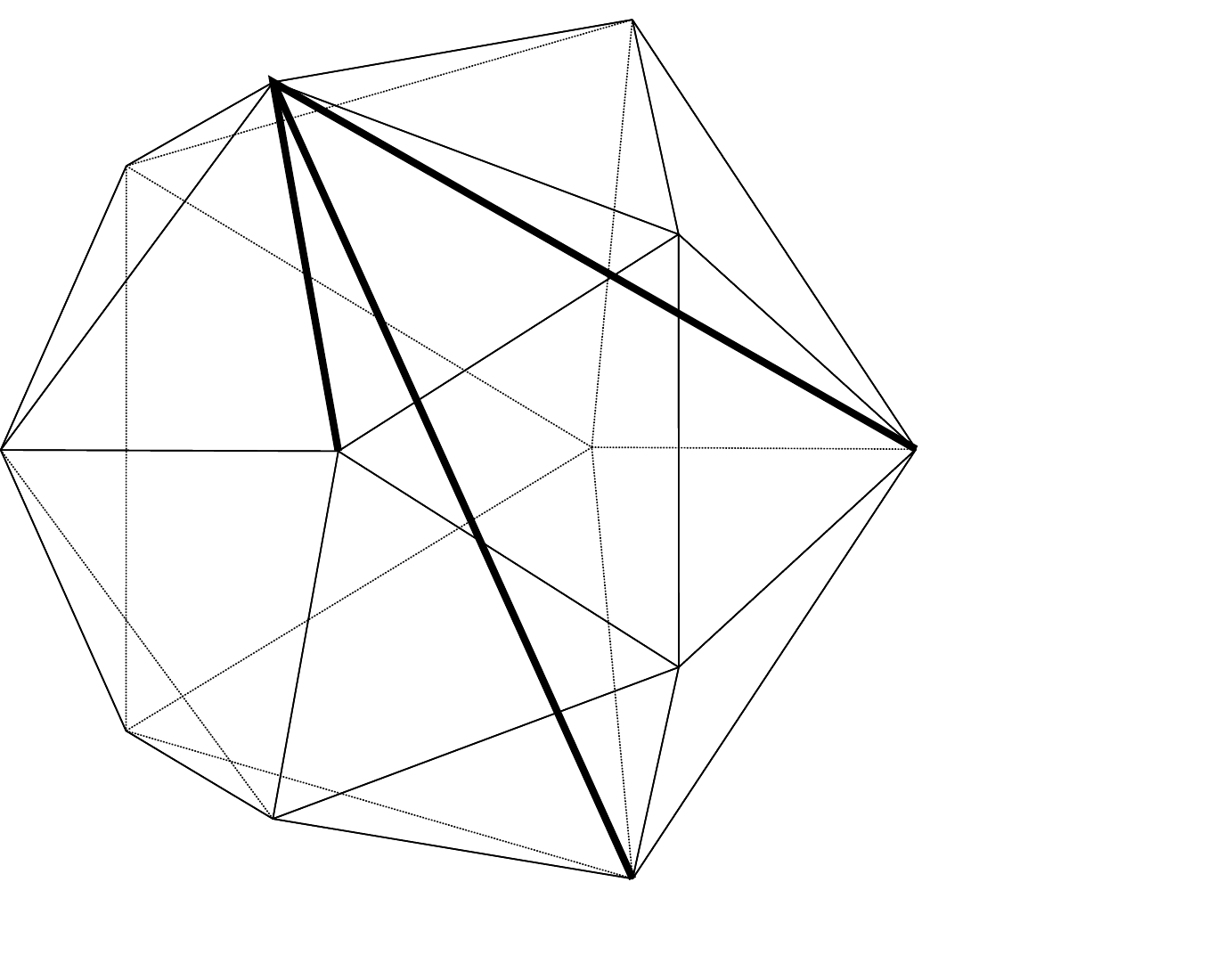
\caption{Representatives of chord orbits in an icosahedron}
\label{ico1}
\end{figure}

\begin{figure}[h]
\centering
\def\svgwidth{0.3\columnwidth}
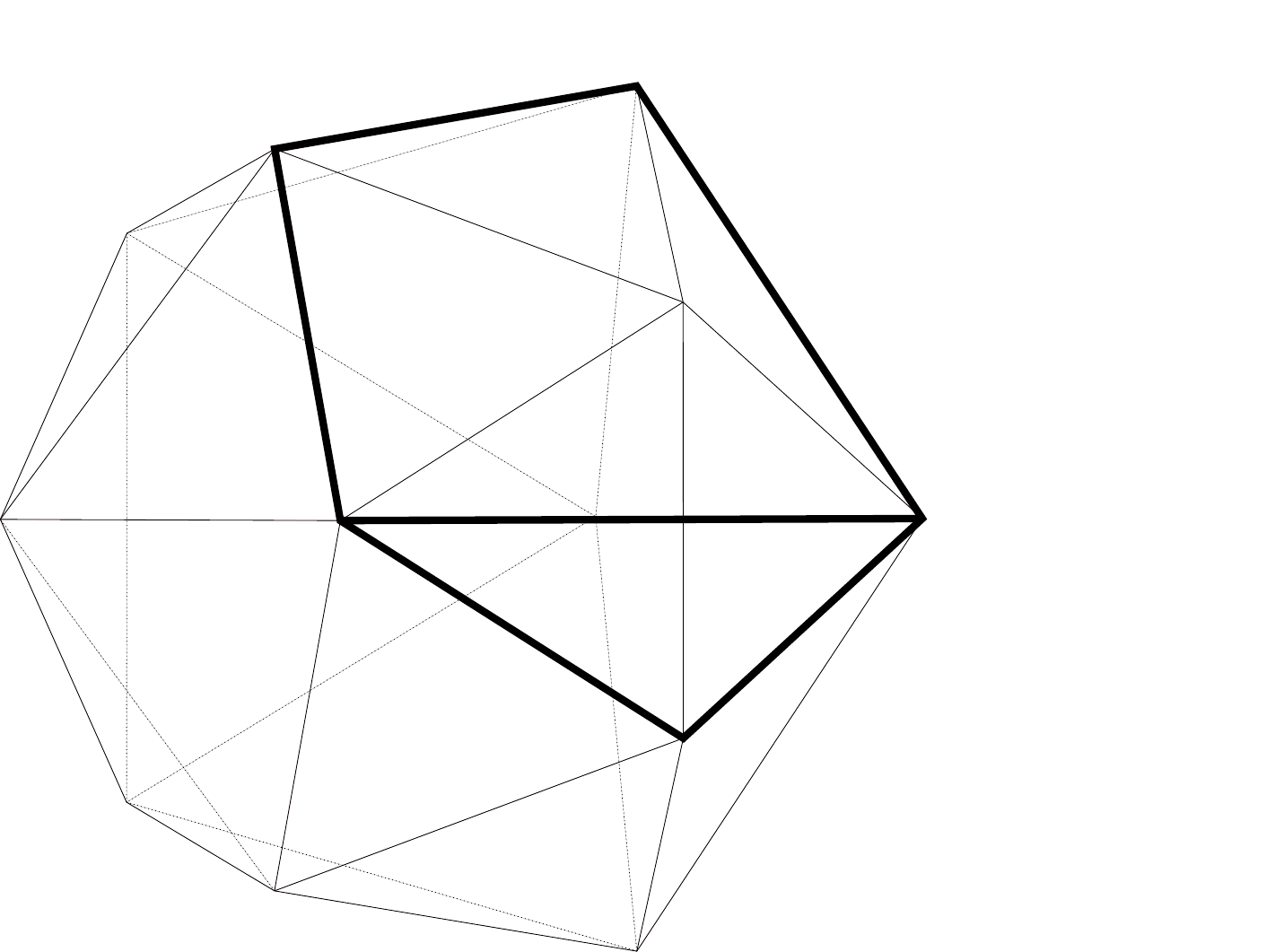
\caption{The pentagon $\bar{x_{h_1} x_{h_2} x_{h_3} x_{h_4} x_{h_5}}$}
\label{ico2}
\end{figure}
From the inequality \ref{ineq},
\begin{align*}
0 &\ge \sum_{i,j=1}^{12} \inner{v_i,v_j} =12 + 60 \cos \alpha_1 + 60 \cos \alpha_2 +12 \cos \alpha_3\\
&=12 +60\left(1-\frac {a_1^2} 2\right)+60\left(1-\frac {a_2^2} 2\right)+12 \cos \alpha_3 \\
&\ge 12 + 60\left(1-\frac {a_1^2} 2\right)+60\left(1-\frac {(2\cos(\pi/5)a_1)^2} 2\right)+12(-1) \\
&=120-30a_1^2 \left(4\cos^2 \frac \pi 5+1\right)
\end{align*}
Hence $a_1^2 \ge 4/(1+4\cos^2(\frac \pi 5)) = 2-2/\sqrt 5$, and so $\alpha_1 \ge \arccos(1/\sqrt 5)$, which is the corresponding angle in the regular icosahedron.
\end{proof}

\begin{proof}[Proof for dodecahedron]
Suppose that $W$ is a dodecahedron. There are six orbits of chords in the regular icosahedron, the representatives of which are $(\bar{x_0},\bar{x_1})$, $(\bar{x_0},\bar{x_2})$, $(\bar{x_0},\bar{x_3})$, $(\bar{x_0},\bar{x_4})$, $(\bar{x_0},\bar{x_5})$ and $(\bar{x_0},\bar{x_6})$, as shown in Figure \ref{ico1}. Denote the lengths of the corresponding chords $(v_0,v_i)$ in $C_0(Y)$ as $a_i$, and the angles they make with the circumcenter as $\alpha_i$.

\begin{figure}[h]
\centering
\def\svgwidth{0.25\columnwidth}
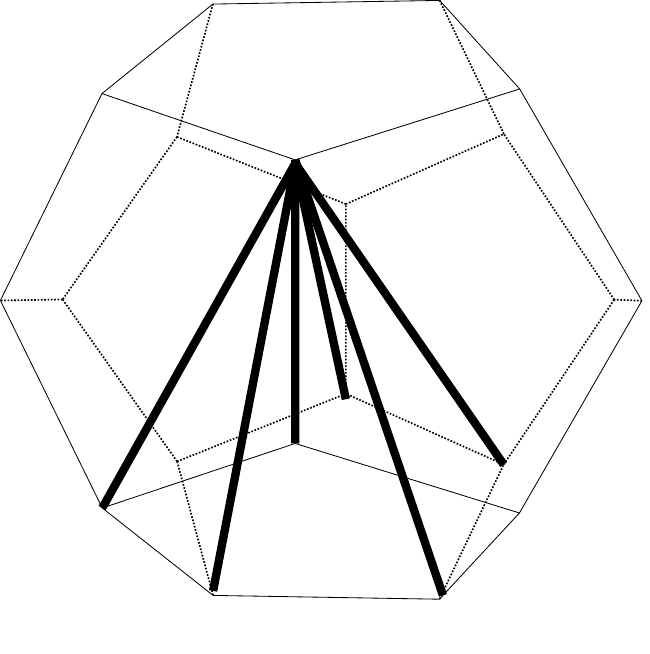
\caption{Representatives of chord orbits in a dodecahedron}
\label{dodec1}
\end{figure}

In Figure \ref{dodec2}, the set of vertices $\{v_{h_1},v_{h_2},v_{h_3},v_{h_4},v_{h_5}\}$ is stabilized by the subgroup $C_5$, hence  $a_2 =d(v_{h_1},v_{h_3}) \le 2 \cos(\pi/5) a_1$.

\begin{figure}[h]
\centering
\def\svgwidth{0.3\columnwidth}
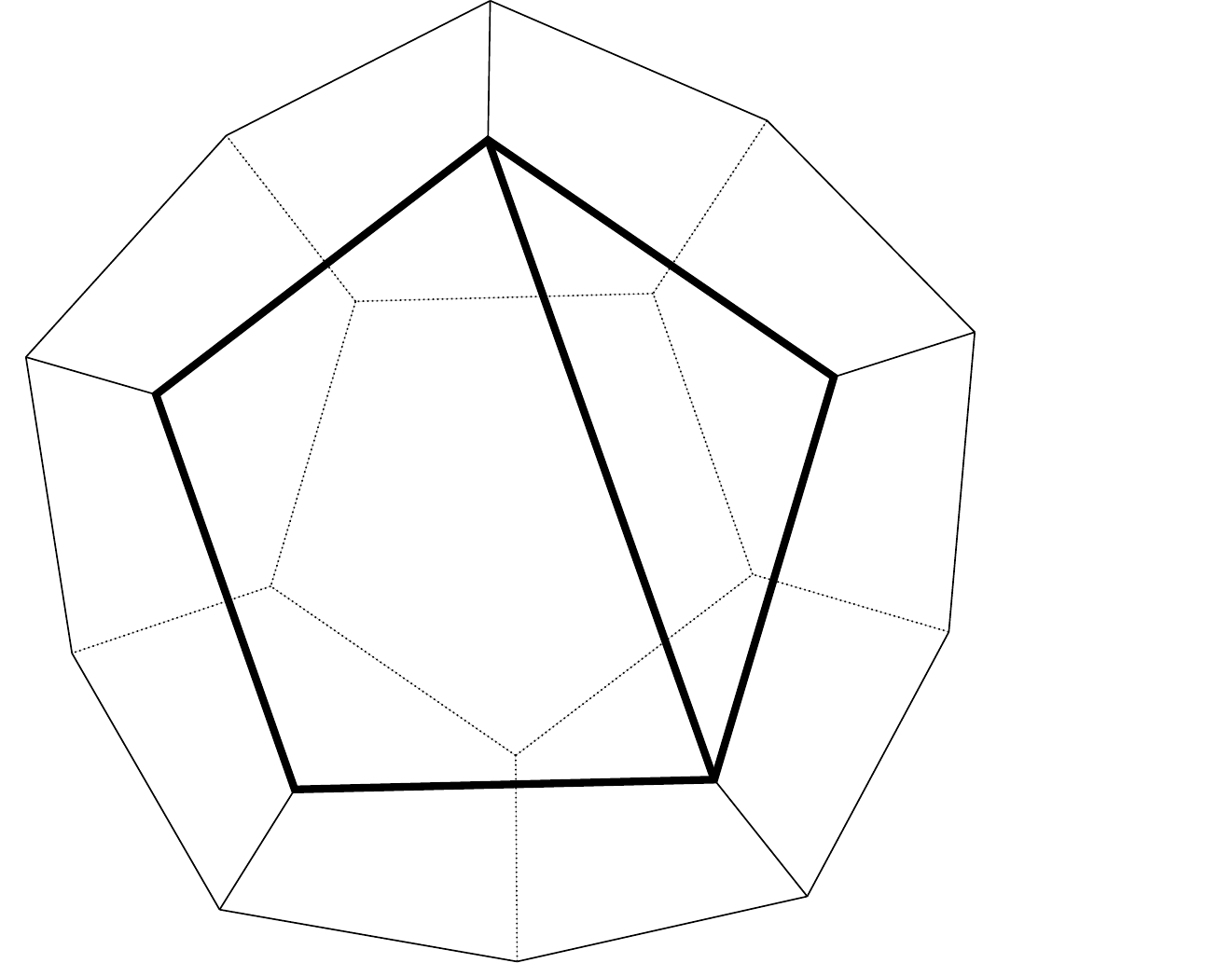
\caption{The pentagon $\bar{x_{h_1}} \bar{x_{h_2}} \bar{x_{h_3}} \bar{x_{h_4}} \bar{x_{h_5}}$}
\label{dodec2}
\end{figure}

The same holds for the set of vertices  $\{v_{i_1},v_{i_2},v_{i_3},v_{i_4},v_{i_5}\}$, as shown in Figure \ref{dodec3}, hence $a_5 =d(v_{i_1},v_{i_3}) \le 2 \cos(\pi/5) a_2 \le  4 \cos^2(\pi/5) a_1$.

\begin{figure}[h]
\centering
\def\svgwidth{0.35\columnwidth}
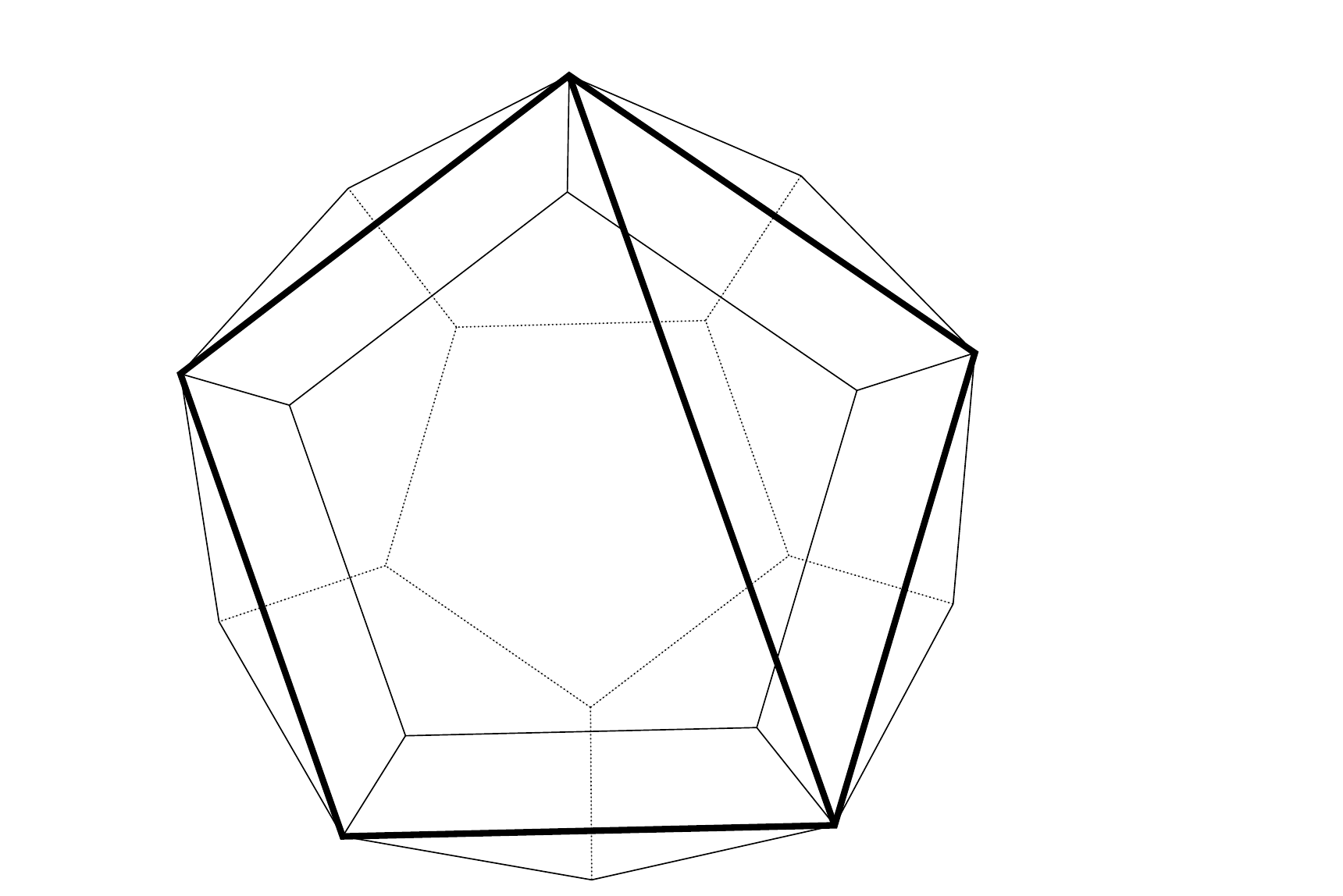
\caption{The pentagon $\bar{x_{i_1}} \bar{x_{i_2}} \bar{x_{i_3}} \bar{x_{i_4}}\bar{x_{i_5}}$}
\label{dodec3}
\end{figure}

 In Figure \ref{dodec4}, since all the four chords of  the quadrilateral $\{v_{j_0},v_{j_1},v_{j_2},v_{j_3}\}$ have equal length, we may apply the parallelogram inequality for CAT(0) space to it and get

\begin{align*}
a_3^2+a_4^2&=d(v_{j_0},v_{j_2})^2 + d(v_{j_1},v_{j_3})^2 \\
 &\le d(v_{j_0},v_{j_1})^2 + d(v_{j_1},v_{j_2})^2+d(v_{j_2},v_{j_3})^2 + d(v_{j_3},v_{j_0})^2 \\
 &= 4 a_2^2 \le 16 \cos \left(\frac \pi 5\right) a_1^2
\end{align*}

\begin{figure}[h]
\centering
\def\svgwidth{0.25\columnwidth}
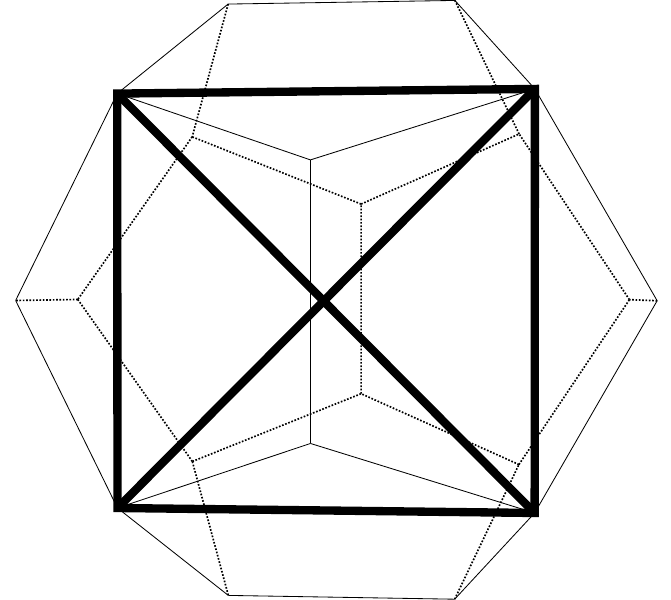
\caption{The quadrilateral $\bar{x_{j_0}}\bar{x_{j_1}} \bar{x_{j_2}} \bar{x_{j_3}} $}
\label{dodec4}
\end{figure}

From the inequality \ref{ineq},
\begin{align*}
0 &\ge \sum_{i,j=1}^{20} \inner{v_i,v_j}\\ &= 20(1 + 3\cos\alpha_1 + 6\cos\alpha_2+3\cos\alpha_3+3\cos\alpha_4+3\cos\alpha_5+\cos\alpha_6)\\
&\ge  20(1 + 3\cos\alpha_1 + 6\cos\alpha_2+3\cos\alpha_3+3\cos\alpha_4+3\cos\alpha_5+(-1))\\
&=20\left(3\left(1-\frac {a_1^2} 2\right)+6\left(1-\frac {a_2^2} 2\right) + 3\left(2-\frac {a_3^2} 2-\frac {a_4^2} 2\right)+3\left(1-\frac {a_2^5} 2\right) \right) \\
&\ge 20\left(18-\frac 3 2 a_1^2 -12 \cos^2\left(\frac\pi 5\right) a_1^2 -24 \cos^2\left(\frac\pi 5 \right) a_1^2-24\cos^4\left(\frac\pi 5\right) a_1^2\right)
\end{align*}
Then $a_1^2 \ge 2-(2\sqrt 5)/3$, so $\alpha_1 \ge \arccos(1-a_1^2/2) = \arccos(\sqrt 5/3)$, which is the corresponding angle in the regular dodecahedron.
\end{proof}
This gives the proof of Theorem \ref{anglepolyhedra} except the last statement.

\subsection{Equality case}
We want to show that for each of these cases, if the equality is attained, then the convex hull of the set of vertices is isometric to the corresponding regular polytope $W$. We will deal with the convex hull in $C_0(Y)$ first, which will imply the same for the convex hull in $X$. To prove the result in $C_0(Y)$, first we show that the convex hulls corresponding to the bounding faces of $W$ are isometric to the faces, secondly, that the subpolytopes which these convex hull make with the origin $o$ are isometric to corresponding ones of $W$, and, lastly, that the union of these subpolytopes is isometric to $W$.

For a set $F$ and a point $p$ in $C_0(Y)$, we define the ``cone'' $\mathrm{Cone}(F,p)$ as the union of $F$ and all geodesic segments joining $x\in F$ to $p$. It need not be homeomorphic to the topological cone of $F$.

\begin{proposition}\label{FlatBaseConeIsometricSide}
Suppose that $P$ is a compact convex subset in $C_0(Y)$ isometric to  $\bar P$ in $\R^k$ for some $k$, the origin $o \not\in P$, and for all $x \in \pa P$, $$d(x,o) = d(\bar x ,\bar c),$$ where $\bar c$ is the origin of $\R^k$. Then $\mathrm{Cone}(P,o)$ is isometric to $\mathrm{Cone}(\bar P, \bar c)$.
\end{proposition}
\begin{proof}
For any points $p_1,p_2 \in P$, extend the segment $[p_1,p_2]$ to one with endpoints in $\pa P$. By assumption the flat triangle formed with this segment and $o$ is isometric to the corresponding one in $\mathrm{Cone}(\bar P, \bar c)$, thus the  flat triangles $\triangle(p_1,p_2,o)$ and $\triangle(\bar {p_1},\bar {p_2}, \bar c)$ are isometric. Any two points in $\mathrm{Cone}(P,o)$ lie in some triangle $\triangle(p_1,p_2,o)$, hence the result.
\end{proof}

The following proposition gives a condition to check when two adjacent flat $k$-dimensional polytopes embedded in a CAT(0) space are not ``folded'' along their intersection, as might be the case if they are in a Euclidean space of dimensional higher than $k$.

\begin{proposition}\label{AdjacentPolyhedra}
Suppose that $P_1$ and $P_2$ are two closed convex subsets in a complete CAT(0) with $P_1 \cap P_2 = F$, and there is an isometry from $P_1 \union P_2$ into $\R^k$ sending $P_1$, $P_2$ and $F$ to $k$-dimensional polytopes $\bar P_1$, $\bar P_2$ and $(k-1)$-dimensional polytope $\bar F$ respectively. Assume that there is a segment $[w_1,w_2]$ in  $P_1\union P_2$ with $w_i$ in $P_i \setminus F$ respectively, such that $[w_1,w_2]$ intersects $F$ at point $p_1$, and the corresponding segment $[\bar{w_1},\bar{w_2}]$ also intersects $\bar F$. Then for any $x_1 \in P_1$ and $x_2 \in P_2$ such that the corresponding segment $[\bar x_1, \bar x_2]$ intersect $\bar F$ at $\bar q$, the union of segments $[x_1, q]\union [q,x_2]$ is a geodesic segment.
\end{proposition}

\begin{proof}
Since  local geodesics are geodesics in CAT(0) spaces, it suffices to prove that the path $[x_1, q]\union [q,x_2]$ is a local geodesic at $q$, so we may move $x_1$ and $x_2$ closer to $q$, so that $\bar q$ is the midpoint of $[\bar{x_1},\bar{x_2}]$.

We start by noting that the segment $[\bar{w_1},\bar{w_2}]$ in $\bar{P_1}\union\bar{P_2}$ intersects $\bar F$   at $\bar{p_1}$. Indeed if they intersect at point $\bar{p'_1}$, then
\begin{align*}
d(w_1,w_2) &\le d(w_1,p'_1)+d(p'_1,w_2) \\
&=d(\bar{w_1},\bar{p'_1})+d(\bar{p'_1},\bar{w_2}) =d(\bar{w_1},\bar{w_2}) \\
&\le d(\bar{w_1},\bar{p_1})+d(\bar{p_1},\bar{w_2}) = d(w_1,p_1)+d(p_1,w_2) =d(w_1,w_2).
\end{align*}
Thus $\bar{p'_1}=\bar{p_1}$.

We claim that if $w_3 \in  P_1$ such that the corresponding segment $[\bar{w_3}, \bar{w_2}]$ intersects $\bar F$ at $\bar p_2$, then $[w_3, p_2] \union [p_2, w_2]$ is a geodesic. If not, then $[w_3,w_2]$ is shorter than $[\bar{w_3}, \bar{w_2}]$, so the comparison triangle of $\triangle(w_1,w_2,w_3)$ will have a smaller comparison angle at $w_1$ than the angle at $\bar{w_1}$ of the flat triangle $\triangle(\bar{w_1}, \bar{w_2},\bar{w_3})$. Since the comparison angle is no less than the original angle, this implies that $\angle_{w_1}(w_2,w_3) < \angle_{\bar{w_1}}(\bar{w_2},\bar{w_3})$, but $$\angle_{w_1}(w_2,w_3) = \angle_{w_1}(p_1,w_3) =\angle_{\bar{w_1}}(\bar{p_1},\bar{w_3}) = \angle_{\bar{w_1}}(\bar{w_2},\bar{w_3}),$$ a contradition. Hence the claim.

Now we assume that $p_1$ is the midpoint of $[w_1,w_2]$ by shortening  $[w_1,w_2]$, then  $\bar p_1$ is also the midpoint of $[\bar{w_1}, \bar{w_2}]$. Take points $w_1, w_2, \ldots, w_7=x_1, w_8=x_2$, where $w_j$ are alternatively in $P_1\setminus F$ or $P_2\setminus F$ for $1\le j \le 8$, with the following properties: $[\bar{w_j}, \bar{w_{j+1}}]$ intersect $\bar F$ at $\bar{p_j}$; the piecewise geodesic path $[\bar{p_1}, \bar{p_3}]\union[\bar{p_3}, \bar{p_5}]\union[\bar{p_5}, \bar{p_7}]$ is contained in $\bar F$;  $[\bar{w_3},\bar{w_4}]$ and $[\bar{w_5},\bar{w_6}]$ have  midpoints $\bar{p_3}$ and $\bar{p_5}$ and are parallel to $[\bar{w_1},\bar{w_2}]$ and $[\bar{w_7},\bar{w_8}]$ respectively; $\bar{w_3},\bar{w_4}, \bar{w_5},\bar{w_6}$ are contained in a ball separated by $\bar F$ . (The last property can be satisfied by shortening $[\bar{w_3},\bar{w_4}]$ and $[\bar{w_5},\bar{w_6}]$ if needed.) With these properties, $\bar{p_2}$ and $\bar{p_6}$ lie on $[\bar{p_1}, \bar{p_3}]$ and $[\bar{p_5}, \bar{p_7}]$ respectively, and $\bar{p_4}$ lie in the intersection of the ball and $\bar F$. So we can apply the claim successively to points $w_j, w_{j+1}, w_{j+2}$ to get the result.
\end{proof}
Note that  if we assume instead that $d(w_1,w_2) = d(\bar{w_1},\bar{w_2})$ and $[\bar{w_1},\bar{w_2}]$ intersects $\bar F$, then this implies that $[w_1,w_2]$ intersects $F$, satisfying the condition in this proposition.

\begin{lemma}\label{FlatBaseConeEqualDist}
Let $P$ be a closed convex set  in a complete CAT(0) space. Suppose that there is a cone $\mathrm{Cone}(P,p)$, and a map $f: \mathrm{Cone}(P,p) \to \R^k$, such that $f|_P$ is an isometry, and $d(p,x) = d(f(p),f(x))$ for all $x\in P$, then $f$ is an isometry from $\mathrm{Cone}(P,p)$ to its image.
\end{lemma}
\begin{proof}
For any $x_1,x_2 \in P$, the triangle $\triangle(p,x_1,x_2)$ has equal side lengths as  the flat triangle $\triangle(f(p),f(x_1),f(x_2))$. Since $d(p,x) = d(f(p),f(x))$ for any point on $[x_1,x_2]$, so  $\triangle(p,x_1,x_2)$ is flat. The result follows.
\end{proof}

In the proof of the last statement of Theorem \ref{anglepolyhedra} for these cases we will assume that $W$ has radius 1.

\begin{proof}[Proof for orthoplex]
Suppose $W$ is a $k$-dimensional orthoplex. We prove by induction on $k$. When $k=2$ it follows from Corollary \ref{polyineq}. Assume it holds for dimension$k-1$. If the equality holds for the $k$-dimensional orthoplex, then $d(v_i, v_{k+i}) =2$ for all $1\le i\le k$. Removing a pair of opposing points $v_k, v_{2k}$ from the set of vertices, the convex hull $K$ of the remaining vertices is a flat $(k-1)$-dimensional orthoplex by induction. Take any $p\in K$, then $p$ is a convex combination of the $2(k-1)$ vertices, and for any of them, say $v_i$, we have  $\inner{v_k, v_i} = 0$. Thus by concavity of the inner product, we have $$\inner{v_k, p} \ge 0,$$ implying that the angle $\angle_o(v_k,p) \le \frac \pi 2.$ Similarly $\angle_o(v_{2k},p) \le \frac \pi 2.$
The equality case implies that $d(v_k,v_{2k}) = 2$, i.e. $\angle_o(v_k,v_{2k})=\pi$, so by the triangle inequality for angles,
$$\angle_o(v_k,p) = \angle_o(v_{2k},p) = \frac \pi 2.$$

Now $\triangle(v_k,p,o)$ is a flat right-angled triangle, so is congruent to the comparison triangle $\triangle(x_k,\bar p, \bar c)$ in $W$, hence $d(v_k,p) = d(x_k,\bar p)$. By Lemma \ref{FlatBaseConeEqualDist}, the cone formed with vertex $v_k$ and base the  flat $k-1$-dimensional orthoplex is isometric to the upper half part of $W$. Likewise, the cone formed with vertex $v_{2k}$ and base the  flat $(k-1)$-dimensional orthoplex is isometric to the lower  half part of $W$ (whereby upper and lower we mean positive and negative $k$-th coordinate). As $d(v_k,v_{2k}) = 2 = d(x_k, x_{2k})$, this pair of vertices satisfies the condition of Proposition \ref{AdjacentPolyhedra}, so the union of these two half-orthoplexes in $C_0(Y)$ is isometric to $W$.
\end{proof}

\begin{lemma}\label{GlueIsomAdjPolyhedra}
Let $P$ be a closed subset in a complete CAT(0) space made up of flat convex subpolytopes $P_1,\cdots, P_j$of dimension $k$, $P'$ be a flat convex polytope made up of convex subpolytopes $P'_1,\cdots,P'_j$. Suppose that there is a bijective map $f$ from $P$ to $P'$ such that  $f$ maps $P_i$ isometrically to $P'_i$ for any $i$, and the union of any two adjacent subpolytopes in $P$ isometrically to two adjacent subpolytopes in $P'$, where two subpolytopes are adjacent when they have a $(k-1)$-dimensional intersection. Then $f$ is an isometry from $P$ to $P'$.
\end{lemma}
\begin{proof}
For any $x$, $y$ in $P$, consider the segment $[f(x), f(y)]$ in $P'$. Assume first that whenever this segment passes between two subpolytopes of $P'$, these two are adjacent. By the assumption, the preimage of this segment is a local geodesic in $P$, so is a geodesic, hence $d(x,y)=d(f(x),f(y))$. If the assumption does not hold, take a sequence $y_n$ converging to $y$ such that the assumption holds for each $[f(x), f(y_n)]$. This is possible since there are only a finite number of faces of codimension at least 2 between $x$ and $y$. Then $d(x,y_n) = d(f(x),f(y_n))$, and passing to the limit we have $d(x,y)=d(f(x),f(y))$. Hence the result.
\end{proof}

\begin{proof}[Proof for dodecahedron]
Suppose $W$ is a dodecahedron with radius 1. The equality and Corollary \ref{polyineq} imply that every five vertices that correspond to a pentagonal face of $W$ form the vertices of a flat pentagon in $C_0(Y)$. In addition, triangle formed by $o$ and any edge of the pentagon is flat, and so by the equality is isometric to the comparison triangle in $W$. Then Proposition \ref{FlatBaseConeIsometricSide} implies that the cone formed by these flat pentagon with $o$ is isometric to the cone in $W$. The equality case also implies that for any two adjacent cones, the distance of the two vertices that are directly opposite to the common edge of the pentagons satisfy the condition for Poposition \ref{AdjacentPolyhedra}. Hence by Lemma \ref{GlueIsomAdjPolyhedra}, these vertices form a flat dodecahedron isometric to $W$.
\end{proof}

\begin{proof}[Proof for icosahedron]
Suppose $W$ is an icosahedron. Take any four vertices, say $v_1, v_2, v_3, v_4$, that correspond to vertices, say $\bar{x_1}, \bar{x_2}, \bar{x_3}, \bar{x_4}$, of two adjacent triangular faces of $W$, of which $v_2$ and $v_3$ are the common vertices. Let $m$ and $\bar m$ be the midpoints of the segments $[v_2,v_3]$ and $[\bar{x_2}, \bar{x_3}]$ respectively.  Then by CAT(0) inequality  $d(v_i,m) \le d(\bar{x_i}, \bar m)$ for $i=1,4$. Since $d(v_i,o) = d(\bar{x_i},\bar c)$ and $d(m,o)= d(\bar m,\bar c)$, so $\angle_o(v_i,m) \le \angle_{\bar c} (\bar{x_i},\bar m)$. But $$\angle_{\bar c} (\bar{x_1},\bar{x_4})=\angle_{\bar c} (\bar{x_1},\bar m)+\angle_{\bar c} (\bar{x_4},\bar m) ,$$ and since $d(v_1,v_4) = d(\bar{x_1},\bar{x_4})$ we have $\angle_{ o} (v_1,v_4)=\angle_{\bar c} (\bar{x_1},\bar{x_4})$, so $$\angle_{o} (v_1,m) + \angle_{o} (v_4,m) \ge \angle_{ o} (v_1,v_4)=\angle_{\bar c} (\bar{x_1},m)+\angle_{\bar c} (\bar{x_4},m).$$ Hence $\angle_{o}(v_i,m) = \angle_{\bar c}(\bar{x_i}, \bar m)$, and so $d(v_i,m) = d(\bar{x_1}, \bar m)$, meaning that $\triangle(v_1,v_2,v_3)$ and $\triangle(v_2,v_3, v_4)$ are flat. Now the pair $v_1$ and $v_4$ satisfy the condition for Proposition \ref{AdjacentPolyhedra}, so by the same reasoning as above we see that these vertices form a flat icosahedron isometric to $W$.
\end{proof}

\begin{proof}[Proof for hypercube]
Suppose $W$ is a hypercube of dimension $k$ and radius 1. The statement is proved for $k=2$. Assume that the statement is true for dimension $k-1$. The stabilizers of the $(k-1)$-dimensional faces are isomorphic to the symmetric group of $(k-1)$-dimensional hypercube. Since lengths of the edges $a_1$ and of the diagonals $a_{k-1}$ of these faces satisfy the equality $a_{k-1} = \sqrt{k-1}a_1$, by Proposition \ref{EqualPolyhedron}, which will be proved below, these faces are isometric to a $(k-1)$-dimensional hypercube with side length $a_1$. Denote any of these hypercubes by $H$. Since the edges of $H$ form  flat triangles with $o$, which are isometric to the respective ones in $W$, then by Proposition \ref{FlatBaseConeIsometricSide} each of its 2-dimensional faces form a flat cone with $o$ isometric to those in $W$. Applying Proposition \ref{FlatBaseConeIsometricSide} successively, each of its $j$-dimensional faces form a flat cone with $o$ for $2\le j \le k-1$ isometric to those in $W$. Then since the distances of pairs of vertices from adjacent $(k-1)$-dimensional hypercubes are equal to the distances of the corresponding ones in $W$, by Proposition \ref{GlueIsomAdjPolyhedra} all these vertices form a flat $k$-dimensional hypercube isometric to $W$.
\end{proof}

This completes the proof of Theorem \ref{anglepolyhedra}. Now we  give a condition that the convex hull of the vertices in the CAT($\kappa$) space $X$ is  isometric to   a regular polytope in $M^2_\kappa$. Recall that  function $\operatorname{sn}:\R \to \R$ is defined by
$$
\operatorname{sn}_\kappa (x):=
                                   \begin{cases}
                                     \sin(\sqrt\kappa x)/\sqrt \kappa & \mbox{if }\kappa > 0; \\
                                     x, & \mbox{if }\kappa = 0; \\
                                     \sinh(\sqrt{-\kappa} x)/\sqrt {-\kappa} & \mbox{if }\kappa < 0.
                                   \end{cases}
$$

We will frequently use this version of the cosine inequality for CAT($\kappa$) spaces.
\begin{lemma}[\cite{LangSchroeder} Lemma 1.3]\label{cosinelaw}
Let $X$ be a CAT($\kappa$) space,$x,y,z\in X$. Let $a=d(y,z)$, $b=d(x,z)$, $c=d(x,y)$, and $\gamma=\angle_z(x,y)$. Assume that $ a,b <\pi/\sqrt\kappa$ and $c \le \pi/\sqrt\kappa$ if $\kappa>0$. Then
$$\sn^2_\kappa \frac c 2 \ge \sn_\kappa^2 \frac{a-b} 2 + \sn_\kappa a \sn_\kappa b \sin^2 \frac \gamma 2.$$
If $X= M^k_\kappa$ then equality holds.
\end{lemma}

Given any of the stated regular polytopes $W$, let $G$ be the symmetry group of the regular polytope. Let $\alpha$ be the lower bound of angle given by Theorem \ref{anglepolyhedra}. In a complete CAT($\kappa$) space $X$ with $G$  acting on a set of points $S=\{x_i\}$ in the same way as the vertices $\{\bar{x_i}\}$ of the regular polytope, and $\operatorname{rad} S < \pi/(2\sqrt\kappa)$ if $\kappa>0$. Let $r=\operatorname{rad} S$,  $a$ be the distance of a pair of points in $S$ corresponding to an edge in $W$.

\begin{proposition}\label{EqualPolyhedron}
Suppose that for any pair of distinct points  $x_i,x_j\in S$,
$$\frac{\operatorname{sn}_\kappa (d(x_i,x_j)/2)} {\sin {({\beta}/2)}}  = \frac{\operatorname{sn}_\kappa (a/2)} {\sin {( {\alpha}/2)}} $$
where $\beta$ is the angle that $\bar{x_i},\bar{x_j}$ make at the circumcenter in $W$. Then the convex hull of $S$ is isometric to a regular polytope of radius $r$ in $M^k_\kappa$.
\end{proposition}
\begin{proof}
For each of the $W$ that we consider, there is a pair of antipodal vertices $x_{i'}, x_{j'}$. Using the equality we have
$$\frac{\operatorname{sn}_\kappa (a/2)} {\sin {( {\alpha} /2)}}  = \operatorname{sn}_\kappa (d(x_{i'}, x_{j'}) /2)\le \operatorname{sn}_\kappa r $$

Let $c$ be the circumcenter of $S$, $\alpha_1$ be the angle of any edge of $S$ made at $c$, then by Theorem \ref{anglepolyhedra} $\alpha_1 \ge \alpha$. Use Lemma \ref{cosinelaw} on the triangle formed by the edge and $c$, we have
$$\operatorname{sn}_\kappa \frac a 2 \ge \operatorname{sn}_\kappa r \sin \frac {\alpha_1} 2 \ge  \operatorname{sn}_\kappa r \sin \frac {\alpha} 2.$$
So these two inequalities are actually an equality. Therefore, by Theorem \ref{anglepolyhedra}  the projection of $S$ to the tangent cone $C_0(\bar{S_c{X}})$ at the circumcenter has a convex hull isometric to the flat regular polytope with radius 1. The intersection of this convex hull and the unit sphere centered at $o$ in $C_0(\bar{S_c{X}})$ is isometric to a Euclidean unit sphere $S^{k-1}$.

We now follow the approach  in the proof of Theorem A in \cite{LangSchroeder}. Consider the $\kappa$-cone  $C_\kappa$ over the space of directions $\bar{S_c{X}}$. Define $f:B(c,r)\to C_\kappa$ by mapping $c$ to the origin $o$ in $C_\kappa$ and $x\in X$ to the point in $C_\kappa$ in the direction of $x$ from $c$ and with  distance   $d(x,c)$. By the CAT($\kappa$) inequality, for any points $x,y \in X$, the comparison triangle $\triangle(\bar c, \bar x, \bar y)$ in $M^2_\kappa$ has an angle at $\bar c$  no less than  $\angle_c(x,y) = \angle_o(f(x),f(y))$, and $\triangle(o,f(x),f(y))$ is isometric to its comparison triangle in $M^2_\kappa$ by definition of a $\kappa$-cone. So
 $$d(f(x),f(y)) \le d(\bar x,\bar y) = d(x,y),$$
i.e. $f$ is 1-Lipschitz.

 For any distinct points $x_i,x_j$ we have
$$\operatorname{sn}_\kappa (d(f(x_i),f(x_j)) /2) = \operatorname{sn}_\kappa (r) \sin (\angle_c(x_i,x_j)/2) =  \operatorname{sn}_\kappa (r) \sin (\beta/2)$$
where $\beta$ is the angle that $\bar{x_i},\bar{x_j}$ make at the circumcenter in $W$, which equals $\angle_c(x_i,x_j)$ by Theorem \ref{anglepolyhedra}. From the equality assumption in the proposition we also get
$$\operatorname{sn}_\kappa (d(x_i,x_j) /2) = \frac {\operatorname{sn}_\kappa (a/2)}{\sin (\alpha/2) } \sin (\beta/2) =  \operatorname{sn}_\kappa (r) \sin (\beta/2)$$
Hence $f$ is an isometry on $S$.

Let $S^*$ be a maximal set in $B(c,r)$ such that $S\subset S^*$ and  $f$ maps $S^*$ isometrically into the convex hull of $f(S)$. We show that $S^*$ is the convex hull of $S$ in the same way as in \cite{LangSchroeder}. For any geodesic $\sigma:[0,1]\to X$ with $\sigma(0),\sigma(1)\in S^*$, $f$ being 1-Lipschitz means that $f\circ \sigma$ is at most as longer as $\sigma$, and since $f$ is an isometry on $S^*$, $f\circ \sigma$  is a geodesic. For any point $x\in S^*$ and $t\in[0,1]$, $d(f(x),f(\sigma(t))) \le d(x,\sigma(t))$ as $f$ is 1-Lipschitz, while applying CAT($\kappa$) inequality on triangle $\triangle(x,\sigma(0),\sigma(1))$ and its comparison $\triangle(f(x),f(\sigma(0)),f(\sigma(1)))$, we have  $d(f(x),f(\sigma(t))) \ge d(x,\sigma(t))$. Therefore $d(f(x),f(\sigma(t)))= d(x,\sigma(t))$, so $\sigma(t)\in S^*$, i.e. $S^*$ is convex. Since $f$ maps geodesic segments in $S^*$ to geodesic segments in $f(S^*)$, so $f(S^*)$ is convex, hence $f(S^*)$ is the convex hull of $f(S)$, and $S^*$ is the convex hull of $S$.  $f(S)$ is contained in the $\kappa$-cone of the unit sphere $S^{k-1} \subset \bar{S_c{X}}$, which is the model space $M^k_\kappa$, and the angle made by any two of the vertices in $f(S)$ equals the angle made by the corresponding vertices in $W$, hence the convex hull of their images is isometric to a regular polytope of radius $r$ in $M^k_\kappa$.
\end{proof}

\begin{theorem}\label{EqualPolyhedronIneq}
Let $X$ and $S$ be as stated immediately before Proposition \ref{EqualPolyhedron}. We have
$$\sn_\kappa \frac a 2 \ge \sn_\kappa r \sin\frac \alpha 2, $$
with equality iff the convex hull of $S$ is isometric to a regular polytope of radius $r$ in $M^k_\kappa$.
\end{theorem}
The inequality is proved in Proposition \ref{EqualPolyhedron}, so we only have to show the equality case of the theorem. To do that we need the following results. The following lemma can be proved in essentially the same way as Corollary \ref{polyineq} and so the proof is omitted.
\begin{lemma}\label{kappapoly}
Let $X$ be a complete CAT($\kappa$) space, $g$ an isometry of order $n\ge 4$ on $X$. For any point $x \in X$ not fixed by $g$ and such that $\{g^i\cdot x\}_{i=1}^n$ has radius less than $\pi/\sqrt{\kappa}$ if $\kappa >0$,
$$\sn_\kappa (d(g^2\cdot x, x)/2) \le 2\cos(\pi/n)\sn_\kappa (d(g\cdot x,x)/2).$$
\end{lemma}
The lemma below is a substitute for the CAT(0) parallelogram inequality in CAT($\kappa$) space.
\begin{lemma}\label{kapparhombus}
Let $(x_1,x_2,x_3,x_4)$ be a rhombus in a complete CAT($\kappa$) space $X$, i.e. all edges $(x_i,x_{i+1})$ have equal length, and assume that its radius is less than $\pi/(2\sqrt\kappa)$ if $\kappa > 0$. Then
$$\sn_\kappa (d(x_1,x_3)/2)\sn_\kappa (d(x_2,x_4)/2) \le 2 \sn_\kappa^2 (d(x_1,x_2)/2).$$
\end{lemma}

\begin{proof}
Consider the comparison triangles $\triangle(\bar{x_1},\bar{x_2},  \bar{x_3})$ in $M^2_\kappa$. Let $\bar m$ be the midpoint of $[\bar{x_1},  \bar{x_3}]$. Let $a=d(\bar{x_1},  \bar{x_2})$, $b =d(\bar{x_1},  \bar m)$, $c=d(\bar{x_2},\bar m)$. Applying Lemma  \ref{cosinelaw} to the two triangles $\triangle(\bar{x_1},\bar{x_2},  \bar{m})$ and $\triangle(\bar{x_2},  \bar{x_3},\bar{m})$,
$$\sn_\kappa^2 \frac a 2 = \sn_\kappa^2 \frac {b-c} 2 + \sn_\kappa  b \sn_\kappa   c  \sin^2 (\frac 1 2 \angle_{\bar m}(\bar{x_1},\bar{x_2})),$$
$$\sn_\kappa^2 \frac a 2 = \sn_\kappa^2 \frac {b-c} 2 + \sn_\kappa  b \sn_\kappa   c \sin^2 (\frac 1 2 \angle_{\bar m}(\bar{x_2},\bar{x_3}))$$
Summing up, noting that  $\angle_{\bar m}(\bar{x_1},\bar{x_2}) + \angle_{\bar m}(\bar{x_2},\bar{x_3}) = \pi$,
$$ 2 \sn_\kappa^2 \frac a 2 = 2 \sn_\kappa^2 \frac {b-c} 4 + \\sn_\kappa  b \sn_\kappa   c \ge \sn_\kappa  b \sn_\kappa   c.$$
Now $d(x_1,x_3) =2b$, $d(x_2,x_4) \le d(x_2,m)+d(m,x_4) \le 2 d(\bar {x_2},\bar m) = 2c$, so we have the result.
\end{proof}

\begin{proof}[Proof of Theorem \ref{EqualPolyhedronIneq}]
When equality holds, we have $\angle_c(x_i,x_j)=\angle_{\bar c}(\bar{x_i},\bar{x_j})$ for all $x_i,x_j\in S$ , so applying Lemma  \ref{cosinelaw} to $\triangle(c,x_i,x_j)$,
$$\sn_\kappa (\frac 1 2 d(x_i,x_j)) \ge \sn_\kappa r \sin  (\frac 1 2\angle_{\bar c}(\bar{x_i},\bar{x_j})) = \frac{\sn_\kappa( a/2)}{\sin(\alpha/2)}\sin  (\frac 1 2\angle_{\bar c}(\bar{x_i},\bar{x_j}))$$
We will now show  the reverse inequality of the above, which  will imply an equality, then by Proposition \ref{EqualPolyhedron} we have the result.

\textit{$W$ is a $k$-dimensional orthoplex.} For any pair of points $x_i,x_j$ in $S$ that corresponds to an antipodal pair of vertices in $W$,
$$\sn_\kappa \frac a 2 =\sn_\kappa r \sin\frac \alpha 2 \ge \sn_\kappa (\frac 1 2 d(x_i,x_j)) \sin\frac \alpha 2.$$
(Likewise we have similar inequalities in the other cases of $W$ for pairs of points which corresponds to  antipodal pairs of vertices, and we will not repeat.)

\textit{$W$ is a $k$-dimensional hypercube.} For each $2\le m\le k$, consider the set $S_m$ that corresponds to a $m$-dimensional hypercube in $W$. This set is invariant under the $m$-dimensional hypercube symmetry group action, so applying Theorem \ref{anglepolyhedra} and Lemma \ref{cosinelaw} to the triangle formed by an edge in $S_m$ and the circumcenter of $S_m$,
$$\sn_\kappa \frac a 2 \ge \sn_\kappa r_m \sin(\arccos(1-2/m)/2) = \frac 1 {\sqrt m} \sn_\kappa r_m \ge  \frac 1 {\sqrt m} \sn_\kappa (\frac {a_m} 2),$$
where $r_m$ is the radius of $S_m$.

\textit{$W$ is an icosahedron.}
Refer to Figure \ref{ico2}. Let $a=d(x_{h_1},x_{h_2})$ and $b=d(x_{h_1},x_{h_3})$. Since the pentagon $(x_{h_1},x_{h_2},x_{h_3},x_{h_4},x_{h_5})$ is stabilized by an isometry of order 5, applying Lemma \ref{kappapoly}
$$\sn_\kappa \frac b 2 \le 2 \cos(\frac\pi 5)\sn_\kappa \frac a 2 = \sin\frac{\angle_{\bar c}(\bar{x_{h_1}},\bar{x_{h_3}})}{2} \sn_\kappa r .$$

\textit{$W$ is a dodecahedron.}
Refer to Figure \ref{dodec2}. Let $a=d(x_{h_1},x_{h_2})$ and $b_1=d(x_{h_1},x_{h_3})$. With the same reason as above we apply  Lemma \ref{kappapoly} to get
$$\sn_\kappa \frac {b_1} 2 \le 2 \cos(\frac\pi 5)\sn_\kappa \frac a 2 = \sin\frac{\angle_{\bar c}(\bar{x_{h_1}},\bar{x_{h_3}})}{2} \sn_\kappa r.$$
Next, refer to Figure \ref{dodec3}. Let $b_2=d(x_{i_1},x_{i_3})$. Again as before we obtain
$$\sn_\kappa \frac {b_2} 2 \le 2 \cos(\frac\pi 5)\sn_\kappa \frac {b_1} 2 = \sin\frac{\angle_{\bar c}(\bar{x_{i_1}},\bar{x_{i_3}})}{2}\sn_\kappa r,$$
where we have used the equality for $b_1$ and $r$.
Finally, refer to Figure \ref{dodec4}. Let $b_3=d(x_{j_0},x_{j_2})$, $b_4=d(x_{j_1},x_{j_3})$. We know already that
$$\sn_\kappa \frac {b_1} 2 = \sn_\kappa r \sin\frac{\angle_{\bar c}(\bar{x_{j_1}},\bar{x_{j_0}})}{2},$$
and we also know that
\begin{equation}\label{eq:b3b4}
\begin{aligned}
\sn_\kappa \frac {b_3} 2 &\ge \sn_\kappa r \sin\frac{\angle_{\bar c}(\bar{x_{j_0}},\bar{x_{j_2}})}{2},\\
\sn_\kappa \frac {b_4} 2 &\ge \sn_\kappa r \sin\frac{\angle_{\bar c}(\bar{x_{j_1}},\bar{x_{j_3}})}{2}.
\end{aligned}
\end{equation}
From Lemma \ref{kapparhombus},
$$\sn_\kappa \frac {b_3} 2 \sn_\kappa \frac {b_4} 2  \le 2\sn_\kappa^2 \frac {b_1} 2.$$
But
$$2\sn_\kappa^2 \frac {b_1} 2 = \sn_\kappa^2 r \sin\frac{\angle_{\bar c}(\bar{x_{j_0}},\bar{x_{j_2}})}{2}\sin\frac{\angle_{\bar c}(\bar{x_{j_1}},\bar{x_{j_3}})}{2},$$
so the inequalities \ref{eq:b3b4} are actually equalities.
\end{proof}
\bibliographystyle{alpha}
\bibliography{reference}
\end{document}